\documentclass[12pt]{amsart}
\usepackage{amsthm,amssymb}
\usepackage[all]{xy}
\newcommand{\Q}{\mathbb{Q}}
\newcommand{\Z}{\mathbb{Z}}
\newcommand{\F}{\mathbb{F}}

\newcommand{\Gal}{{\rm Gal}}
\newcommand{\GL}{{\rm GL}}

\newcommand{\ord}{{\rm ord}}

\usepackage{amsthm}
\newtheorem{thm}{Theorem}
\newtheorem*{thmnonum}{Theorem}
\newtheorem{lem}[thm]{Lemma}
\newtheorem{cor}[thm]{Corollary}

\newtheorem*{ack}{Acknowledgements}
\newtheorem*{rem}{Remark}
\newcommand{\im}{{\rm im}~}

\newcommand{\Aut}{{\rm Aut}}
\newcommand{\End}{{\rm End}}

\begin{document}

\title[Arboreal images]{Uniform bounds on the image of the arboreal Galois representations attached to non-CM elliptic curves}
\author{Michael Cerchia}
\author{Jeremy Rouse}
\begin{abstract}
  Let $\ell$ be a prime number and let $F$ be a number field and $E/F$ a non-CM elliptic curve with a
  point $\alpha \in E(F)$ of infinite order. Attached to the pair
  $(E,\alpha)$ is the $\ell$-adic arboreal Galois representation
  $\omega_{E,\alpha,\ell^{\infty}} : \Gal(\overline{F}/F) \to \Z_{\ell}^{2} \rtimes
  \GL_{2}(\Z_{\ell})$
  describing the action of $\Gal(\overline{F}/F)$ on points
  $\beta_{n}$ so that $\ell^{n} \beta_{n} = \alpha$. We give
  an explicit bound on the index of the image of $\omega_{E,\alpha,\ell^{\infty}}$
  depending on how $\ell$-divisible the point $\alpha$ is,
  and the image of the ordinary $\ell$-adic Galois representation.
  The image of $\omega_{E,\alpha,\ell^{\infty}}$ is connected with the density of primes
  $\mathfrak{p}$ for which $\alpha \in E(\F_{\mathfrak{p}})$ has order
  coprime to $\ell$.
\end{abstract}

\maketitle

\section{Introduction and Statement of Results}
\label{intro}

Let $F$ be a number field, $E/F$ an elliptic curve, and $\alpha \in
E(F)$ a point of infinite order. For each prime $\mathfrak{p}$ of $F$
of good reduction for $E$, $E(\mathbb{F}_{\mathfrak{p}})$ is a finite
abelian group and so $\alpha \in E(\mathbb{F}_{\mathfrak{p}})$ has
finite order.  It is natural to ask how often $\alpha$ has odd order
or even order (or more generally how often the order of $\alpha$ is
coprime to any fixed prime $\ell$). It seems reasonable to guess that
$\alpha$ has odd order ``half'' the time.  However, in
\cite{JonesRouse}, Rafe Jones and the second author determined that
for $E : y^{2} + y = x^{3} - x$ and $\alpha = (0,0)$, the density of
primes $p$ for which $\alpha \in E(\F_{p})$ has odd order is
$\frac{11}{21}$.

It is elementary to see that $\alpha \in E(\mathbb{F}_{\mathfrak{p}})$
has order coprime to $\ell$ if and only for all $n \geq 1$
there is some $\beta_{n} \in E(\mathbb{F}_{\mathfrak{p}})$ so that
$\ell^{n} \beta_{n} = \alpha$. This connects the order of
$\alpha \in E(\mathbb{F}_{\mathfrak{p}})$ with Galois-theoretic properties
of the preimages $\beta_{n}$ of $\alpha$ under multiplication by powers of $\ell$. These are governed by the arboreal Galois representation $\omega_{E,\alpha,\ell^{\infty}}$. To define this, we need some notation.

Let $E[\ell^{n}]$ denote the set of points of order dividing $\ell^{n}$ on $E$
and define $\rho_{E,\ell^{n}} : \Gal(F(E[\ell^{n}])/F) \to \Aut(E[\ell^{n}]) \cong \GL_{2}(\Z/\ell^{n} \Z)$
to be the usual mod $\ell^{n}$ Galois representation. Let $T_{n} = F(E[\ell^{n}])$
and $T_{\infty} = \bigcup_{n=1}^{\infty} F(E[\ell^{n}])$.

The representations $\rho_{E,\ell^{n}}$ are compatible, and if we let
$T_{\ell}(E) = \varprojlim E[\ell^{n}]$ be the $\ell$-adic Tate module
of $E$, we get a representation $\rho_{E,\ell^{\infty}} :
\Gal(T_{\infty}/F) \to \Aut(T_{\ell}(E))$. Fixing an isomorphism
$T_{\ell}(E) \to \Z_{\ell}^{2}$, we can view the image of $\rho$ as a
subgroup of $\GL_{2}(\Z_{\ell})$. Serre \cite{SerrePropriete} has
proven that if $E/F$ does not have complex multiplication, the image
of $\rho_{E,\ell^{\infty}}$ is a finite index subgroup of
$\GL_{2}(\Z_{\ell})$. Given an $\alpha \in E(F)$, fix a sequence of
points $\beta_{1}, \beta_{2}, \ldots$ so that $\ell \beta_{1} =
\alpha$ and $\ell \beta_{n} = \beta_{n-1}$ for $n \geq 1$. Let $K_{n}
= F(E[\ell^{n}],\beta_{n})$. For each $\sigma \in \Gal(K_{n}/F)$,
$\sigma(\beta_{n})$ is also a preimage of $\alpha$ under
multiplication by $\ell^{n}$ and so $\sigma(\beta_{n}) - \beta_{n} \in
E[\ell^{n}]$. Define $\omega_{E,\alpha,\ell^{n}} : \Gal(K_{n}/F) \to
E[\ell^{n}] \rtimes \Aut(E[\ell^{n}])$ by
$\omega_{E,\ell^{n},\alpha}(\sigma) = (\sigma(\beta_{n}) - \beta_{n},
\sigma|_{E[\ell^{n}]})$. These representations are again compatible
and if we let $K_{\infty} = \bigcup_{n=1}^{\infty} K_{n}$, they give
rise to $\omega_{E,\alpha,\ell^{\infty}} : \Gal(K_{\infty}/F) \to
T_{\ell}(E) \rtimes \Aut(T_{\ell}(E)) \cong \Z_{\ell}^{2} \rtimes
\GL_{2}(\Z_{\ell})$. Theorem~3.2 of \cite{JonesRouse} shows that the
density of primes $\mathfrak{p}$ for which $\alpha \in
E(\F_{\mathfrak{p}})$ has order coprime to $\ell$ only depends on the
image of $\omega_{E,\alpha,\ell^{\infty}}$, and Theorem 5.5 of
\cite{JonesRouse} shows that, in the case that
$\omega_{E,\alpha,\ell^{\infty}}$ is surjective, the density of primes
for which $\alpha$ has order coprime to $\ell$ is $\frac{\ell^{5} -
  \ell^{4} - \ell^{3} + \ell + 1}{\ell^{5} - \ell^{3} - \ell^{2} +
  1}$.

The goal of the present paper is to prove uniformity results about the
image of $\omega_{E,\alpha,\ell^{\infty}}$. In \cite{LT} (see
Definition 4.1) the authors declare a point $\alpha \in E(F)$ to be
\emph{strongly $\ell$-indivisible} if $\alpha+T \not\in \ell E(F)$ for
any torsion point $T \in E(F)$ of $\ell$-power order. This is a
natural primitivity condition to impose. Without this condition, the
index of the image of $\omega$ in $\Z_{\ell}^{2} \rtimes
\GL_{2}(\Z_{\ell})$ can be arbitrarily large, and the corresponding
density can be made very large (by taking $\alpha = \ell^{k} \gamma$
for some large $k$ and $\gamma \in E(F)$) or very small (by taking
$\alpha = \ell^{k} \gamma + T$ for $\gamma \in E(F)$ and $T \in
E[\ell](F)$). Moreover, if $\alpha + T = \ell \gamma$ for $\gamma \in
E(F)$ one can read off the representation
$\omega_{E,\alpha,\ell^{\infty}}$ from that of
$\omega_{E,\gamma,\ell^{\infty}}$ and the choice of $T$.

The action of Galois groups arising from division points of elliptic
curves has been studied extensively in the context of Kummer
theory. If we let $E$ be an elliptic curve over a number field $k$,
and $A$ a subgroup of $E(k)$ invariant under $\End_k(E)$, we have the
following consequential result of Ba\v{s}makov \cite{Bashmakov}: For a
prime $p$, denote by $A_p$ the group of elements $x\in
E(\overline{k})$ such that for some $n\geq 0$ $p^nx\in A$. Denote by
$L_p$ the smallest field containing $k$ over which all the elements of
$A_p$ are rational. Ba\v{s}makov proved that if $A$ is torsion free,
then for almost all $p$, $H^1(\Gal(L_p/k),A_p)=0$. This result was
generalized by Ribet in \cite{MR424823}, who proved that for abelian
varieties of CM-type, the Galois group arising from division of
rational points is as large as possible for all but finitely many
primes. A unified approach to results of this type was given by Ribet
in \cite{Ribet} with the results applying to mod $\ell$ Kummer theory
for an arbitrary commutative, connected algebraic group. There are
also $\ell$-adic versions of these results that are in the
literature. See for example, the paper of Bertrand \cite{Bertrand},
and Appendix 2 of the paper of Hindry \cite{Hindry}. The principal
contribution of the present paper is the application of results about
Kummer theory to give uniform bounds on the index of the image of the
Galois representation $\omega_{E,\alpha,\ell^{\infty}}$. Our main
result is the following.

\begin{thm}
\label{main}
Suppose that $F$ is a number field, $\ell$ is a prime,
and $E/F$ is an elliptic curve with ${\rm End}_{\overline{F}}(E) \cong \Z$. Denote $\im \rho_{E,\ell^{\infty}} = G$. Let $d$ be the largest positive integer for which
$\alpha = \ell^{d} \gamma + T$
for some $\gamma \in E(F)$ and some $F$-rational $\ell$-power torsion point $T$. Then the index of $\im \omega_{E,\alpha,\ell^{\infty}}$
in $\Z_{\ell}^{2} \rtimes \GL_{2}(\Z_{\ell})$ is at most
$\ell^{2d+2r+s} \cdot |\GL_{2}(\Z_{\ell}) : G|$, where
$r$ is the smallest positive integer so that $G$ contains
a matrix $\begin{bmatrix} x & 0 \\ 0 & x \end{bmatrix}$
with $\ord_{\ell}(x-1) = r$, and $s$ is the largest positive integer
such that there is a degree $\ell^{s}$ cyclic isogeny $E \to E'$ defined over $F$.
\end{thm}

\begin{rem}
  The paper \cite{LT}, by Lombardo and Tronto which was finished at nearly the same time as this present paper, contains a similar result to the one above. In particular, Theorem 4.15 gives a bound on the index of the image of
  $\omega_{E,\alpha,\ell^{\infty}}|_{\Gal(K_{\infty}/T_{\infty})}$ in $\Z_{\ell}^{2}$ in terms
  of the parameter $n_{\ell}$ (the smallest positive integer so that the image
  of $\rho_{E,\ell^{\infty}}$ contains all matrices $\equiv I \pmod{\ell^{n_{\ell}}}$.
  The bound given in Theorem 4.15 is $2d + 4n_{\ell}$, while the bound we give
  is $\leq 2d + 3n_{\ell}$. 
\end{rem}

The bound in Theorem~\ref{main} is sharp in some but not all cases.
In general, if one fixes $G \subseteq \GL_{2}(\Z_{\ell})$, it is not
straightforward to generate examples of elliptic curves $E/F$ with
$\im \rho_{E,\ell^{\infty}} = G$ and with the image of
$\omega_{E,\ell^{\infty}}$ unusually small.

In \cite{REU2015}, the
authors show that if $F = \Q$, $\ell = 2$, $\rho_{E,2^{\infty}}$ is
surjective, and $\alpha$ is strongly $2$-indivisible, then
$\omega_{E,\alpha,2^{\infty}}$ is either surjective (in which case the
density of primes $p$ for which $\alpha$ has odd order is $11/21$), or
the image of $\omega_{E,\alpha,2^{\infty}}$ has index $4$ in
$\Z_{2}^{\times} \rtimes \GL_{2}(\Z_{2})$, and the odd order density
is $179/336$. In the context of the main theorem, $d = 0$, $r = 1$, $s
= 0$ and $|\GL_{2}(\Z_{2}) : G| = 1$ and so the bound given by the
main theorem is $4$.

In \cite{LiangRouse}, it is shown that if the
image of $\rho_{E,2^{\infty}} = \left\{ \begin{bmatrix} a & b \\ 0 &
  d \end{bmatrix} : a, d \in \Z_{2}^{\times}, b \in \Z_{2} \right\}$,
and $\alpha$ is strongly $2$-indivisible, then there are $63$
possibilities for the image of $\omega_{E,\alpha,2^{\infty}}$ up to
conjugacy in $\Z_{2}^{2} \rtimes \GL_{2}(\Z_{2})$ and the index
is at most $8$. This corresponds to $d = 0$, $r = 1$ and $s = 1$ in
the main theorem and so the main theorem gives a sharp result in this case.

If $E : y^{2} = x^{3} - 343x + 2401$ and $\alpha =
(0,-49)$, then $\alpha$ is strongly $2$-indivisible (and so $d =
0$). The image $G$ of $\rho_{E,2^{\infty}}$ is an index $4$ subgroup
of $\GL_{2}(\Z_{2})$ and corresponds to the modular curve $X_{2a}$ of
\cite{RZB}. The curve $E$ has no cyclic isogenies defined over $\Q$
and so $s = 0$, while $G$ contains $5I$ but does not contain $mI$ for
any $m \equiv 3 \pmod{4}$.  Thus, $r = 2$. One can explicitly compute
that if $\beta_{2}$ is a point so that $4 \beta_{2} = \alpha$, then
$\beta_{2} \in E(\Q(E[8]))$ and this implies that the image of
$\omega_{E,\alpha,2^{\infty}}$ has index $64$ in $\Z_{2}^{2} \rtimes
\GL_{2}(\Z_{2})$. We have that $64 = 2^{2d+2r+s} \cdot
|\GL_{2}(\Z_{2}) : G|$, so the bound is sharp in this case.

If $E/\Q$ is an elliptic with $2$-adic image equal to $X_{238a}$ of
\cite{RZB}, and $\alpha \in E(\Q)$ which is strongly $2$-indivisible,
the parameters from the main theorem are $d = 0$, $r = 4$, $s = 1$ and
$|\GL_{2}(\Z_{2}) : G| = 96$. Thus, the bound given is $2^{2 \cdot 4 +
  1} \cdot 96 = 49152$. As will be explained later, in this case, the
index of the image of $\omega$ inside $\Z_{2}^{2} \rtimes
\GL_{2}(\Z_{2})$ can be at most $3072$ and in this instance the main theorem
is not sharp.

Finally, for an elliptic curve $E/\Q$ with $2$-adic image equal to
$X_{243g}$ of \cite{RZB} and strongly $2$-indivisible $\alpha$,
the parameters in the main theorem are $d = 0$, $r = 3$, $s = 4$
and $|\GL_{2}(\Z_{2}) : G| = 96$. Thus, the bound given is
$2^{2 \cdot 3 + 4} \cdot 96 = 98304$. The authors are not aware of any reason
the index of the image of $\omega$ in $\Z_{2}^{2} \rtimes \GL_{2}(\Z_{2})$
cannot be this large, but we have not found an example
of curve and a point $\alpha$ for which this bound is achieved. 

The parameters $r$ and $s$ in the theorem above depend only on
the structure of $T_{\ell}(E)$ as a Galois module, and hence only on $G$.
We can strengthen the above result using known results about the
image of $\rho_{E,\ell^{\infty}}$.
\begin{thmnonum}
  If $m$ is a fixed positive integer, then as $E$ ranges over all non-CM
  elliptic curves $E/F$, where $[F : \Q] = m$, there are only finitely many
  possibilities for the image of $\rho_{E,\ell^{\infty}}$ in $\GL_{2}(\Z_{\ell})$.
\end{thmnonum}
A complete proof of this theorem can be found in \cite{ClarkPollack} (see Theorem 2.3(a)), relying on work of Serre \cite{SerrePropriete}, Abramovich \cite{Abramovich}, and Frey \cite{Frey}. A more general statement is the main theorem of \cite{CadoretTamagawa}.

The above result immediately yields the following corollary.
\begin{cor}
\label{uniform}
Fix a positive integer $m$ and a prime $\ell$. Then there is a constant $C(m,\ell)$ with the
following property.  For all number fields $F$ with $[F : \Q] = m$ and
for all pairs of non-CM curves and points $E/F$ and $\alpha \in E(F)$
which are strongly $\ell$-indivisible, the image of $\omega_{E,\alpha,\ell^{\infty}}$ in $\Z_{\ell}^{2} \rtimes \GL_{2}(\Z_{\ell})$ has index $\leq C(m,\ell)$.
\end{cor}

We prove Theorem~\ref{main} by showing that the fields $F(\beta_{n})$ and $T_{n} = F(E[\ell^{n}])$ are ``approximately disjoint''. The failure of disjointness of these
fields is related to the exponent of the cohomology group $H^{1}(T_{m}/F, E[\ell^{n}])$. The order of this cohomology group can be shown to be uniformly
bounded (as a function of $m$ and $n$) in terms of the image of $\rho_{E,\ell^{\infty}}$.

\begin{ack}
  This work represents joint work done when the first author was a master's student at Wake Forest University. The authors wish to thank the anonymous referee
  for many helpful comments that improved the paper.
\end{ack}

\section{Proof of Theorem~\ref{main}}

Let $G$ be a topological group and let $M$ be a right topological
$G$-module. Define as usual the group $H^{1}(G,M)$ of continuous $1$-cocycles
$\xi : G \to M$ modulo $1$-coboundaries. We will use the following
Lemma of Sah (see \cite{Lang}, page 212, for a proof).

\begin{lem}
\label{sah}
  If $\alpha \in Z(G)$, then the endomorphism $(\alpha - 1)$ of $H^{1}(G,M)$
  is the zero map.
\end{lem}

Let $F$ be a number field, $E/F$ an elliptic curve and $\alpha \in
E(F)$ be a non-torsion point. Let $d$ be the smallest positive integer
so that $\alpha = \ell^{d} \gamma + T$ for some $\gamma \in E(F)$ and
$T$ an $F$-rational $\ell$-power torsion point.  Define a sequence of
points $\beta_{1}, \beta_{2}, \ldots$ so that $\ell \beta_{1} =
\alpha$ and $\ell \beta_{n} = \beta_{n-1}$.  As in
Section~\ref{intro}, define $T_{n} = F(E[\ell^{n}])$, $T_{\infty} =
\bigcup_{n=1}^{\infty} T_{n}$, $K_{n} = F(E[\ell^{n}],\beta_{n})$ and
$K_{\infty} = \bigcup_{n=1}^{\infty} K_{n}$ and fix an isomorphism
$T_{\ell}(E) \to \Z_{\ell}^{2}$. Let $\omega_{E,\alpha,\ell^{\infty}} : \Gal(K_{\infty}/F) \to
\Z_{\ell}^{2} \rtimes GL_{2}(\Z_{\ell})$ be the arboreal Galois representation
and $\kappa_{E,\alpha,\ell^{\infty}} : \Gal(K_{\infty}/T_{\infty}) \to
T_{\ell}(E)$ be the $\ell$-adic Kummer map, the first coordinate of
$\omega_{E,\alpha,\ell^{\infty}}$, namely $\kappa(\sigma) =
(\sigma(\beta_{1}) - \beta_{1}, \sigma(\beta_{2})-\beta_{2},\ldots)
\in \underset{i}{\varprojlim} E[\ell^{i}] = T_{\ell}(E)$. In
\cite{Bertrand}, Bertrand shows that the image of $\kappa$ has finite
index in $T_{\ell}(E)$ using similar cohomological arguments to those
of Ribet in \cite{Ribet}.

\begin{proof}[Proof of Theorem~\ref{main}]
  We fix $G = \im \rho_{E,\ell^{\infty}} \subseteq \GL_{2}(\Z_{\ell})$ and
  let $r$ and $s$ be as defined in the statement of Theorem~\ref{main}.
  We show that if $m > r + d$, then $\beta_{m} \not\in E(T_{\infty})$.
  Suppose to the contrary that $\beta_{m} \in E(T_{\infty})$ and define
  $\xi : \Gal(T_{\infty}/F) \to E[\ell^{m}]$ by $\xi(\sigma) = \sigma(\beta_{m}) - \beta_{m}$. This is a $1$-cocycle and gives rise to an element of
  $H^{1}(T_{\infty}/F, E[\ell^{m}])$. If $g = \begin{bmatrix} x & 0 \\ 0 & x \end{bmatrix} \in G$ with $\ord_{\ell}(x-1) = r$, then $g-I$ kills every element
  of $H^{1}(T_{\infty}/F, E[\ell^{m}])$. This is the same as multiplication by
  $x-1$. Since $H^{1}(T_{\infty}/F, E[\ell^{m}])$ is an $\ell$-group, it follows
  that $\ell^{r} \xi$ is a 1-coboundary. Thus, there is some $T \in E[\ell^{m}]$
  so that $\ell^{r} \sigma(\beta_{m}) - \ell^{r} \beta_{m} = \sigma(T) - T$
  for all $\sigma \in \Gal(T_{\infty}/F)$. This implies that
  $Q = \ell^{r} \beta_{m} - T \in E(F)$, which leads to a contradiction
  because $\alpha = \ell^{m-r} Q + \ell^{m-r} T$, and
  $\alpha \in E(F)$ and $Q \in E(F)$ implies that $\ell^{m-r} T \in E(F)$
  is an $\ell$-power torsion point.

It follows that there is some $\sigma \in \Gal(K_{\infty}/T_{\infty})$
so that $\sigma(\beta_{r+d+1}) \ne \beta_{r+d+1}$. Consider
$\vec{v} = \kappa(\sigma) \in \Z_{\ell}^{2}$ as a row vector. The reduction of $\vec{v}$ in $(\Z/\ell^{r+d+1} \Z)^{2}$ is nonzero and therefore, the minimal $\ell$-adic valuation of the coordinates is at most $r+d$. Write $\vec{v} = \ell^{t} \vec{p}$, where $\vec{p}$ is a primitive vector (one where not both entries are $\equiv 0 \pmod{\ell}$) and $t \leq r+d$.
Since the image of
$\kappa \subseteq \Z_{\ell}^{2}$ has a natural action of $G = \im \rho_{E,\ell^{\infty}}$ on it, the image of $\kappa$ contains $\ell^{t} \vec{p} g$ for all $g \in G$. Let
$S \subseteq \Z_{\ell}^{2}$ be the smallest subgroup containing $\vec{p} g$
for all $g \in G$. Since $G$ is open in $\GL_{2}(\Z_{\ell})$, there is a positive
integer $k$ so that $G$ contains all $g \equiv I \pmod{\ell^{k}}$. A straightforward calculation shows that for any $\vec{q} \equiv \vec{p} \pmod{\ell^{k}}$
there exists some $g \equiv I \pmod{\ell^{k}}$ so that $g \vec{p} = \vec{q}$.
Let $k' \geq k$ be the smallest non-negative integer so that $S$ contains
all vectors $\equiv \vec{0} \pmod{\ell^{k'}}$ and define
$\overline{S} = \{ \vec{v} \bmod \ell^{k'} : \vec{v} \in S \}$.
Identify $\overline{S}$ with a subgroup of $E[\ell^{k'}]$. Note that $\overline{S}$ must be cyclic (because it does not contain all vectors $\equiv 0 \pmod{\ell^{k'-1}}$ by the minimality of $k'$), and it must also have order $\ell^{k'}$
since it contains a primitive vector and is cyclic. Finally, $\overline{S}$
is stable under the action of $\Gal(T_{\infty}/F)$. From Proposition III.4.12 and Remark III.4.13.2 of \cite{Silverman} it follows that there is an $F$-rational cyclic isogeny $\phi : E \to E'$ for some elliptic curve $E'/F$ whose kernel is $\overline{S}$. We have then that
\[
|\Z_{\ell}^{2} : S| = |(\Z/\ell^{k'} \Z)^{2} : \overline{S}| = |\overline{S}| = \ell^{k'}
\]
is less than or equal to the maximum degree of a cyclic $\ell$-power isogeny $\phi : E \to E'$ defined over $F$, namely $\ell^{s}$.  Since the image of $\kappa$ contains $\ell^{r+d} \cdot S$, it follows that the index of the image of $\kappa$ is $\leq \ell^{2d+2r+s}$ and moreover that the image of $\kappa$
contains all $\vec{v} \equiv 0 \pmod{\ell^{d+r+s}}$.

Next, we show that the image of $\omega_{E,\alpha,\ell^{\infty}}$ has finite index
in $\Z_{\ell}^{2} \rtimes \GL_{2}(\Z_{\ell})$. Let $n$ be the number of elements
in the set $\{ \vec{y} \bmod \ell^{d+r+s} : (\vec{y},g) \in \im \omega \text{ for some } g \in \GL_{2}(\Z_{\ell}) \}$, and let $\vec{y}_{1}, \ldots, \vec{y}_{n}$
  be elements in $\Z_{\ell}^{2}$ that reduce mod $\ell^{d+r+s}$ to the elements
  of that set. For any $(\vec{y},g) \in \im \omega$, there is some
  $\vec{y}_{i}$ so that $\vec{y} \equiv \vec{y}_{i} \pmod{\ell^{d+r+s}}$. Consequently, $(\vec{y}_{i},g) = (\vec{y},g) * (\vec{y}_{i} - \vec{y}, I)$ is
  in the image of $\omega_{E,\alpha,\ell^{\infty}}$. 

  Let $H = \{ g \in \GL_{2}(\Z_{\ell}) : (\vec{0},g) \in \im \omega \}$. We will prove that $H$ is a finite index subgroup of $\GL_{2}(\Z_{\ell})$. Define $m = \max \{k, d+r+s \}$ and let $\Gamma(\ell^{m}) = \{ g \in \GL_{2}(\Z_{\ell}) : g \equiv I \pmod{\ell^{m}} \}$. In particular, we will show that if $g_{1} \equiv g_{2} \pmod{\ell^{m}}$
  and $(\vec{y}_{i}, g_{1})$ and $(\vec{y}_{i}, g_{2})$ are both in
  the image of $\omega$, then $g_{1}$ and $g_{2}$ are in the same left coset
  of $H \cap \Gamma(\ell^{m})$. Letting $\vec{x} = -\vec{y}_{i} (-g_{1}^{-1} g_{2} + I)$ we have that $\vec{x} \equiv 0 \pmod{\ell^{m}}$ and hence $(\vec{x},I) \in \im \omega$. Therefore
\[
(\vec{y}_{i}, g_{1})^{-1} * (\vec{y}_{i}, g_{2}) * (\vec{x}, I)
= (\vec{y}_{i} (-g_{1}^{-1} g_{2} + I), g_{1}^{-1} g_{2}) * (-\vec{y}_{i} (-g_{1}^{-1} g_{2} + I), I) = (\vec{0}, g_{1}^{-1} g_{2}) \in \im \omega
\]
and thus $g_{1}^{-1} g_{2} \in H \cap \Gamma(\ell^{m})$. Since every
$g \in \Gamma(\ell^{m})$ has the property that $(\vec{y}_{i},g) \in
\im \omega$ for some $i$, it follows that $[\Gamma(\ell^{m}) : H \cap
  \Gamma(\ell^{m})] \leq n$ and so $[\GL_{2}(\Z_{\ell}) : H] \leq
    [\GL_{2}(\Z_{\ell}) : H \cap \Gamma(\ell^{m})] \leq
    [\GL_{2}(\Z_{\ell}) : \Gamma(\ell^{m})] [\Gamma(\ell^{m}) : H \cap
      \Gamma(\ell^{m})] \leq n \cdot |\GL_{2}(\Z/\ell^{m} \Z)|$ is
    finite. It follows that $H$ is an open subgroup of
    $\GL_{2}(\Z_{\ell})$ and hence it contains a neighborhood of the
    identity matrix. So there is some integer $v \geq m$ so that every
    $g \equiv I \pmod{\ell^{v}}$ is contained in $H$. Finally, if
    $\vec{x} \equiv 0 \pmod{\ell^{v}}$ and $g \equiv I
    \pmod{\ell^{v}}$ then $(\vec{x},I)$ and $(\vec{0},g)$ are both in
    the image of $\omega$ and so the image of $\omega$ contains
    $(\vec{0},h) * (\vec{x},I) = (\vec{x},h)$ and hence all
    $(\vec{x},h)$ with $\vec{x} \equiv 0 \pmod{\ell^{v}}$ and $h
    \equiv I \pmod{\ell^{v}}$. This implies that the image of $\omega$
    is the full preimage in $\Z_{\ell}^{2} \rtimes \GL_{2}(\Z_{\ell})$
    of the image of $\omega_{E,\alpha,\ell^{v}}$.  The map $\kappa_{v}
    : \Gal(K_{v}/T_{v}) \to E[\ell^{v}]$ given by $\kappa_{v}(\sigma)
    = \sigma(\beta_{v}) - \beta_{v}$ has an image that contains the
    mod $\ell^{v}$ reduction of the image of $\kappa$, and
    $\Gal(T_{v}/F)$ has index $|\GL_{2}(\Z_{\ell}) : G|$ inside
    $\GL_{2}(\Z/\ell^{v} \Z)$. It follows from this that
\[
|\Z_{\ell}^{2} \rtimes \GL_{2}(\Z_{\ell}) : \im \omega| \leq
\ell^{2d+2r+s} |\GL_{2}(\Z_{\ell}) : G|, 
\]
as desired.
\end{proof}

Regarding the example of elliptic curves with $2$-adic image equal
to $X_{238a}$ of \cite{RZB}, the image of $\rho_{E,2^{\infty}}$
contains $\begin{bmatrix} 17 & 0 \\ 0 & 17 \end{bmatrix}$
but it does not contain any $\begin{bmatrix} x & 0 \\ 0 & x \end{bmatrix}$
with $\ord_{2}(x-1) < 4$. The bound on the exponent of $H^{1}(T_{\infty}/F,E[2^{m}])$ given by Sah's lemma is $16$. However, one can use Magma \cite{Magma}
to compute that $H^{1}(\Gal(\Q(E[2^{8}])/\Q),E[8])$ has exponent $4$ and
use the inflation-restriction sequence to show that
\[
  H^{1}(\Gal(\Q(E[2^{n}])/\Q),E[8]) \cong H^{1}(\Gal(\Q(E[2^{8}])/\Q),E[8])
\]
for $n \geq 8$. Therefore the main theorem does not give a sharp bound
(because Sah's lemma does not give an optimal bound).

\bibliographystyle{plain}
\bibliography{references}

\begin{thebibliography}{10}

\bibitem{Abramovich}
Dan Abramovich.
\newblock A linear lower bound on the gonality of modular curves.
\newblock {\em Internat. Math. Res. Notices}, (20):1005--1011, 1996.

\bibitem{Bashmakov}
M.I. Ba\v{s}makov.
\newblock Un th\'{e}or\`eme de finitude sur la cohomologie des courbes
  elliptiques.
\newblock {\em C. R. Acad. Sci. Paris S\'{e}r. A-B}, 270:A999--A1001, 1970.

\bibitem{Bertrand}
D.~Bertrand.
\newblock Galois representations and transcendental numbers.
\newblock In {\em New advances in transcendence theory ({D}urham, 1986)}, pages
  37--55. Cambridge Univ. Press, Cambridge, 1988.

\bibitem{REU2015}
Alexi Block~Gorman, Tyler Genao, Heesu Hwang, Noam Kantor, Sarah Parsons, and
  Jeremy Rouse.
\newblock The density of primes dividing a particular non-linear recurrence
  sequence.
\newblock {\em Acta Arith.}, 175(1):71--100, 2016.

\bibitem{Magma}
Wieb Bosma, John Cannon, and Catherine Playoust.
\newblock The {M}agma algebra system. {I}. {T}he user language.
\newblock {\em J. Symbolic Comput.}, 24(3-4):235--265, 1997.
\newblock Computational algebra and number theory (London, 1993).

\bibitem{CadoretTamagawa}
Anna Cadoret and Akio Tamagawa.
\newblock A uniform open image theorem for {$\ell$}-adic representations, {II}.
\newblock {\em Duke Math. J.}, 162(12):2301--2344, 2013.

\bibitem{ClarkPollack}
Pete~L. Clark and Paul Pollack.
\newblock Pursuing polynomial bounds on torsion.
\newblock {\em Israel J. Math.}, 227(2):889--909, 2018.

\bibitem{Frey}
Gerhard Frey.
\newblock Curves with infinitely many points of fixed degree.
\newblock {\em Israel J. Math.}, 85(1-3):79--83, 1994.

\bibitem{Hindry}
Marc Hindry.
\newblock Autour d'une conjecture de {S}erge {L}ang.
\newblock {\em Invent. Math.}, 94(3):575--603, 1988.

\bibitem{JonesRouse}
Rafe Jones and Jeremy Rouse.
\newblock Galois theory of iterated endomorphisms.
\newblock {\em Proc. Lond. Math. Soc. (3)}, 100(3):763--794, 2010.
\newblock Appendix A by Jeffrey D. Achter.

\bibitem{Lang}
Serge Lang.
\newblock {\em Fundamentals of {D}iophantine geometry}.
\newblock Springer-Verlag, New York, 1983.

\bibitem{LiangRouse}
Ke~Liang and Jeremy Rouse.
\newblock Density of odd order reductions for elliptic curves with a rational
  point of order 2.
\newblock {\em Int. J. Number Theory}, 15(8):1547--1563, 2019.

\bibitem{LT}
Davide Lombardo and Sebastiano Tronto.
\newblock Explicit {K}ummer theory for elliptic curves, 2019.

\bibitem{MR424823}
Kenneth~A. Ribet.
\newblock Dividing rational points on {A}belian varieties of {${\rm CM}$}-type.
\newblock {\em Compositio Math.}, 33(1):69--74, 1976.

\bibitem{Ribet}
Kenneth~A. Ribet.
\newblock Kummer theory on extensions of abelian varieties by tori.
\newblock {\em Duke Math. J.}, 46(4):745--761, 1979.

\bibitem{RZB}
Jeremy Rouse and David Zureick-Brown.
\newblock Elliptic curves over {$\Bbb Q$} and 2-adic images of {G}alois.
\newblock {\em Res. Number Theory}, 1:Art. 12, 34, 2015.

\bibitem{SerrePropriete}
Jean-Pierre Serre.
\newblock Propri\'{e}t\'{e}s galoisiennes des points d'ordre fini des courbes
  elliptiques.
\newblock {\em Invent. Math.}, 15(4):259--331, 1972.

\bibitem{Silverman}
Joseph~H. Silverman.
\newblock {\em The arithmetic of elliptic curves}, volume 106 of {\em Graduate
  Texts in Mathematics}.
\newblock Springer-Verlag, New York, 1986.

\end{thebibliography}

\end{document}